\documentclass[10pt, a4paper]{article}
\usepackage{algorithm}
\usepackage{algorithmic}
\usepackage{booktabs}
\usepackage{ulem}
\usepackage{hyperref}
\usepackage{amsmath}        % 几乎必备：数学公式
\usepackage{amsfonts}       % 数学字体
\usepackage{amssymb}        % 数学符号
\usepackage{amsthm}
\newtheorem{theorem}{Theorem}[section]
\newtheorem{lemma}[theorem]{Lemma}
\newtheorem{proposition}[theorem]{Proposition}

\newtheorem{corollary}[theorem]{Corollary}
\newtheorem{example}[theorem]{Example}

\newtheorem{remark}[theorem]{Remark}

\begin{document}
\renewcommand{\algorithmicrequire}{ Input:}
\renewcommand{\algorithmicensure}{Output:}
\makeatletter
\title{Algorithms for Modular Parametrizations of Elliptic Curves over $\mathbb{Q}$}

\author{SanMin Wang\\
 Faculty of Science, Zhejiang Sci-Tech University, \\
Hangzhou 310018, P.R. China\\
wangsanmin@hotmail.com}


%    General info
%\Subjclass[2020]{Primary 11G05, 11F03; Secondary 11Y16, 14H05}

%\date{Aug  XX, 2025 and, in revised form, X X, 202X.}

%{\bf Keywords:}Number-theoretic algorithms, Modular curve,  Elliptic curve,  Modular parametrization

\maketitle

\begin{abstract}
Let \( E \) be a complex elliptic curve with conductor \( N \) and   modular invariant \( j(E) \in \mathbb{Q} \). We construct a class of modular polynomials $F_N(x,j)$ that relate  the modular function $x$ on $X_0(N)$ to the $j$-invariant $j$, where $x$ is  obtained by composing  the first coordinate function of $E$ with
the modular parametrization $\varphi: X_0(N) \rightarrow E$.
Using $F_N(x,j)$,  we can precisely determine the poles of $\varphi$,
compute exact values of $\varphi$ at cusps,  and
develop an algorithm for calculating ramification points of $\varphi$.
Moreover, $F_N(x,j)$ yields   an efficient algorithm for computing  the fibres of $\varphi$
over arbitrary points on $E$.
In some sense, $F_N(x,j)$  also provides   a ``total"  formula for computing the minimal polynomial of
the images of  Heegner points on $X_0(N)$ under $\varphi$.
Especially, we  compute  the semi-trace of the image $\varphi ([\frac{{ - 1 + \sqrt { - 3} }}{2}])$  of  the  CM-point   $[\frac{-1 + \sqrt{-3}}{2}]$ on  $X_{0}(389)$,  under the action of a 65-element subgroup of the 260-element  Galois group of  $\mathbb{Q}(\sqrt{-3}, j(389 \cdot \frac{-1 + \sqrt{-3}}{2}))$. Finally,  we associate a point of infinite order in~\( E(\mathbb{Q}) \) with an infinite sequence~$\{ (j(\tau_n), j(N\tau_n)) \}_{n \in \mathbb{Z}^+} $ of algebraic numbers whose degrees are bounded by the degree of~$\varphi$. This provides one seemingly  practicable  approach to addressing  the BSD conjecture.

\end{abstract}

\section{Introduction}
Let \( E \) be a complex elliptic curve with the conductor \( N \) and modular invariant \( j(E) \in \mathbb{Q} \).
It follows from the work of Wiles, Taylor, Breuil, Conrad and Diamond (\cite{Wiles}, \cite{Taylor-Wiles}, \cite{Breuil-Conrad-Diamond-Taylor}) that  there exists a surjective holomorphic map of compact Riemann surfaces from  the modular curve  \(X_0(N)\) to $E$,
\[\varphi: X_0(N) \to E.\]
 The map \(\varphi\) is called a  modular parametrization of \(E\).

Let \(x\) and \(y\) be the coordinate functions on the affine 2-space \(\mathbb{A}^2_{\mathbb{C}}\). Then \(x \circ \varphi\) and \(y \circ \varphi\) define two meromorphic functions on \(X_0(N)\). We denote \(x \circ \varphi\) simply by \(x\) and \(y \circ \varphi\) by \(y\). By the general theory of algebraic function fields, $x$ and the j-invariant $j$ satisfy an algebraic equation \(F_N(x,j) = 0\).

In Section 2, we present an algorithm for constructing the modular polynomial \(F_N(x,j)\). 
The modular function $x$ generally has  poles in $\mathcal{H}$, then we can't use  the evaluation-interpolation method  in \cite{Enge} to determine $F_N(x, j)$.   Currently, only the $q$-expansion method has been successfully implemented.
Define \( J: \mathcal{H} \to \mathbb{C} \) by \( J(\tau) = j(N\tau) \) for any \(\tau \in \mathcal{H}\).
By the same method, we can construct modular polynomials relating \( x \) and \( J \),
\( y \) and \( j \), \( y \) and \( J \), and denote  them  by \( f_N(x,j) \), \( G_N(y,j) \) and \( g_N(y,J) \), respectively.
 For algorithms for constructing modular polynomials, we refer to \cite{Bruinier-Ono-Sutherland,Enge}.

It is known that for elliptic curves \( E \) over $\mathbb{Q}$ with rank at most one, the Birch and Swinnerton-Dyer (BSD) conjecture is valid up to a controlled rational factor \cite{Cohen1}. Therefore, we are particularly interested in elliptic curves with rank greater than 1, among  which \( E_{a1}(389) \)   is  the elliptic curve  of  rank 2 with the smallest conductor.
In this paper, \( E_{XY}(N) \) denotes the \( Y \)-th elliptic curve in the isogeny class \( X \) of conductor \( N \)  in Cremona's elliptic curve database\cite{Cremona}.   Especially,  we have   constructed  \( F_{389}(x, j) \), whose data file in Mathematica occupies approximately 20MB, while the corresponding data file for \( \Phi_{389}(X, Y) \) requires about 900MB of the storage space.

In Section 3, we develop an algorithm to compute the fiber \(\varphi^{-1}(P)\) for any point \( P = (\alpha, \beta) \in \mathbb{C}^2 \) on $E$
using \( F_N(x, j) \) and \( f_N(x, J) \).
Let \(\{M_n\}_{n=1}^{\mu}\) be a complete set of right coset representatives for \(\Gamma_0(N)\) in \(SL_2(\mathbb{Z})\). The fundamental domain $\mathcal{F}_N$ of \(X_0(N)\) is then given by:
\[
\mathcal{F}_N = \bigcup_{n=1}^{\mu} M_n \mathcal{F},
\]
where \(\mathcal{F}\) denotes the standard fundamental domain of \(SL_2(\mathbb{Z})\).
 Since the degree of \(\varphi\) is generally less than the index \(\mu = [SL_2(\mathbb{Z}) : \Gamma_0(N)]\), we have no efficient method to determine a priori which \(M_n\mathcal{F}\) contains points of \(\varphi^{-1}(P)\), and thus no efficient method to compute \(\varphi^{-1}(P)\) without the help of \( F_N(x, j) \) and \( f_N(x, J) \).

Let \( \mathcal{H}^* = \mathcal{H} \cup \mathbb{Q} \cup \{i\infty\} \).  The modular curve for  $\Gamma_0(N)$ is  defined as \( X_0(N) := \Gamma_0(N) \backslash \mathcal{H}^* \),  which is the compactification of  the modular curve  $Y_0(N):= \Gamma_0(N) \backslash \mathcal{H}$.
We denote a point \( \Gamma_0(N)\tau \) in \( X_0(N) \) by \( [\tau] \), sometimes simply as \( \tau \).
We'll  often use  the following expansion form of  $F_N(x,j) $:
\begin{equation}
F_N(x, j)= \sum_{k=0}^K \sum_{l=0}^L c_{k,l} x^k j^l= \sum_{k=0}^K A_k(j)x^k = \sum_{l=0}^L B_l(x)j^l.
\end{equation}

In Section 4, we present an algorithm to compute \( \varphi \)'s poles by  the observation that
$A_K(j(\tau)) = 0$  for $\tau\in \mathcal{H}$ implies that  $ \varphi([M_n\tau]) = \infty$ for some $1 \leqslant n \leqslant \mu$.  This  is an immediate applications of \( F_N(x,j) \).  The poles of  $\varphi$ contain the zeros of Eichler integrals of $f$,  where $f$ is the newform for ${\Gamma _0}(N)$  associated to $E$ by $\varphi$.
Especially,  we give a counterexample to Kodgis's conjecture\cite{Kodgis,Peluse}, demonstrating that not all zeros of Eichler integrals are CM-points.

In Section 5, we present an algorithm to compute the exact values of \( \varphi \)
 at cusps by  the observation that  $\varphi([\tau])=(\alpha ,\beta ) \in {\mathbb{C}^2} $ for $ \tau \in  \mathbb{Q}$
implies that $B_L(\alpha) = 0$.
This is another  immediate applications of \( F_N(x,j) \).

The study of critical points and ramification points of $\varphi$ dates back to the work of Mazur and Swinnerton-Dyer in the 1970s \cite{Mazur-Swinnerton-Dyer}.  In 2005,  Delaunay \cite{Delaunay} proposed an approximate method for computing zeros of the newform $f$ associated to $E$. Subsequently, Chen \cite{Chen} developed two algorithms, namely Poly Relation and Poly Relation-YP, for computing critical polynomials.
In 2016, Brunault \cite{Brunault} described an algorithm for computing the ramification indices of $\varphi$ at cusps and proved that $\varphi$ is unramified at cusps if $f$ has minimal level among its twists by Dirichlet characters. Later in 2018, Corbett and Saha \cite{Corbett-Saha} established a criterion for determining whether a cusp is a critical point of $\varphi$.  This criterion can be  convert  into  a Magma code  to compute the ramification indices of $\varphi$ at cusps.

In Section 6, we develop an algorithm for computing ramification points using the modular polynomials we constructed. Our approach proceeds as follows.
First, we determine a finite set of potential ramification points using $F_N(x,j)$, $f_N(x,J)$, $G_N(y,j)$, $g_N(y,j)$ and their derivatives. Then, we verify whether each candidate point is indeed a ramification point by computing its fiber.

Compared with Delaunay's algorithm, our search for ramification points is in  a finite set. Furthermore, the computational cost of constructing our modular polynomials is significantly lower than that of Chen's Poly Relation algorithm. The main computational bottleneck of our method lies in solving for the intersection points of the relevant curves.

Let $x = (\phi: E \rightarrow E')$ represent a point on $X_0(N)$ via a cyclic $N$-isogeny $\phi$. We say $x$ is a  complex multiplication point (CM-point for short) on $X_0(N)$ if there exist two orders $\mathcal{O}_1$, $\mathcal{O}_2$ in an imaginary quadratic field $K = \mathbb{Q}(\sqrt{D})$ such that $\mathcal{O}_1 = \text{End}(E)$ and $\mathcal{O}_2 = \text{End}(E')$ \cite{Zhang}. If $\mathcal{O}_1 = \text{End}(E) = \text{End}(E')$, we call $x$ a Heegner point of $X_0(N)$ \cite{Gross}.   Let $G$ be  the Galois group of the splitting field $L$  of  $\mathbb{Q}(j(E), j(E'), \sqrt{D})$.  For a subgroup $H$ of $G$, we define the semi-trace of the image $\varphi(x)$ a CM-point $x$ under the action of $H$ as:
\[
P = \sum_{\sigma \in H} \varphi(x)^\sigma \in E(L).
\]

For a Heegner point $x$, we may first compute the minimal polynomial $Q(j)$ of the algebraic number $j(E)$, and then calculate the resultant of $F_N(x,j)$ and $Q(j)$ with respect to $j$,
finally factor this resultant to obtain the minimal polynomial $f(X)$ of  the abscissa of $\varphi(x)$.
The \texttt{heegner\_point} function in Sage  \cite{Sage} requires
approximate computation of the conjugates of  $\varphi(x)$  and then
numerical approximation to determine $f(X)$.
From this viewpoint, $F_N(x,j)$  provides   a ``total"  formula  for computing the minimal polynomial of
the images of  Heegner points on $X_0(N)$ under $\varphi$.

 In Section 7, we  compute  the semi-trace of the image $\varphi ([\frac{{ - 1 + \sqrt { - 3} }}{2}])$  of  the CM-point  $\tau = \frac{-1 + \sqrt{-3}}{2}$ on  $X_{0}(389)$,  under the action of a 65-element subgroup of the 260-element Galois group $G$ of  $\mathbb{Q}\left(\sqrt{-3}, j\left(389 \cdot \frac{-1 + \sqrt{-3}}{2}\right)\right)$.
Currently, one of the obstacles to progress on the BSD conjecture is the lack of a known method to establish connections between Heegner points on $X_0(N)$ and points of infinite order in $E(\mathbb{Q})$ for elliptic curves $E$ of rank $\geq 2$ via modular parametrizations $\varphi\colon X_0(N) \to E$.
The computation in this section  is an attempt to obtain rational points on $E$ using more general CM-points.

In Section 8, we provide the conditions under which $\deg F_N(x,j)<\mu = [SL_2(\mathbb{Z}):\Gamma_0(N)]$
and some additional properties of $F_N(x,j)$. Note that
if the degree of $x$ in $F_N(x,j)$ equals $\mu$, then $F_N(x,j)$ can serve as a defining equation for $X_0(N)$. For related work on defining equations of $X_0(N)$,  see \cite{Yang}.

In Section 9, we  presents a variant of Algorithm 1 that constructs integer-coefficient polynomials $P_1(j,J)$, $Q_1(j,J)$, $P_2(j,J)$, and $Q_2(j,J)$ such that
\[
x = \frac{P_1(j,J)}{Q_1(j,J)} \quad \text{and} \quad y = \frac{P_2(j,J)}{Q_2(j,J)}.
\]
These polynomials appeared in Kolyvagin's seminal work \cite{Kolyvagin}.
However, as pointed out  by Yang  in \cite{Yang},  it is difficult to explicitly write
down these polynomials in general.
By replacing the input $(j,J)$ with Yang-pairs $(X,Y)$, our algorithm can efficiently compute the polynomials $P_1(X,Y)$, $Q_1(X,Y)$, $P_2(X,Y)$, and $Q_2(X,Y)$ to represent $x$ and $y$. 
We notice  that this issue has already been discussed by C. McLarty and  N. D. Elkies in  \cite{McLarty}.

In Section 10, we discusses the correspondence between the points of degree $d$  on the plane model $Z_0(N)$  of  $X_0(N)$  defined by $\Phi_N(X,Y)$ and rational points on the elliptic curve $E$.
In particular, we associate a point of infinite order in~\( E(\mathbb{Q}) \) with an infinite sequence~$\{ (j(\tau_n), j(N\tau_n)) \}_{n \in \mathbb{Z}^+} $ of algebraic numbers whose degrees are bounded by the degree of~$\varphi$. This provides one seemingly  practicable  approach to addressing  the BSD conjecture.

The code that verifies our computations, along with the outputs can be found on:
https://github.com/SanMinWang2025/Algorithms-for-Modular-\\
Parametrizations-of-Elliptic-Curves-over-Q.git.

\section{Algorithm for Constructing Modular Polynomials $F_N(x,j)$}

Let $E$ be an elliptic curve over $\mathbb{Q}$ with conductor $N$. The Weierstrass equation for $E$ is given by
\begin{equation}
y^2 + a_1xy + a_3y = x^3 + a_2x^2 + a_4x + a_6.
\end{equation}
Let $\varphi: X_0(N) \to E$ be a modular parametrization of $E$, and let $f$ be the  normalized newform of weight $2$ on  $\Gamma_0(N)$ associated to $E$.
For any $\tau \in \mathcal{H}^*$,  define
\begin{equation}
\gamma (\tau) := \int_{{i\infty }}^{\tau} 2\pi i f(z) dz,
\end{equation}
which is called the Eichler integrals of $f$.

For simplicity, we assume throughout this paper that $E$ is an optimal elliptic curve with Manin constant $1$.  Let $L$ be the period lattice of $E$.  Then
\[
L = \{\gamma(M\tau) - \gamma(\tau) : M \in \Gamma_0(N)\}.
\]

Let $\overline{\mathbb{C}} = \mathbb{C} \cup \{\infty\}$. For $\tau \in \mathcal{H}^*$, define:
\begin{align}
x(\tau) &:= \wp(\gamma(\tau)) - \frac{b_2}{12},\\
y(\tau)& := \frac{\wp'(\gamma(\tau))}{2} - \frac{a_1\wp(\gamma(\tau))}{2} + \frac{a_1b_2 - 12a_3}{24},
\end{align}
where $b_2 = a_1^2 + 4a_2$, and $\wp$ is the Weierstrass $\wp$-function for the lattice $L$.  Then

\begin{enumerate}
   \item [(i)]$ \varphi ([\tau ]) =(x(\tau), y(\tau)) \in E;$
    \item[(ii)] \( x: \mathcal{H}^* \to \overline{\mathbb{C}} \) and \( y: \mathcal{H}^* \to \overline{\mathbb{C}} \) are modular functions on \( X_0(N) \);

    \item[(iii)] If \( x(\tau) \neq \infty \), then \( x \) is holomorphic at the point \( \tau \);

    \item[(iv)] If \( y(\tau) \neq \infty \), then \( y \) is holomorphic at the point \( \tau \).
\end{enumerate}

The $q$-expansion of $f(\tau)$ can be computed effectively, enabling term-by-term evaluation of $\gamma(\tau)$. Consequently, the above definitions yield an effective procedure for computing $\varphi([\tau])$ when $\tau \in \mathcal{H}$.  For further details, we refer to \cite{Cohen1,Delaunay}.

Recall that \(  \mathbb{C}(j, J) \) is  the field of meromorphic functions on \( X_0(N) \).
Since \( x \in  \mathbb{C}(j, J) \) and the transcendental degree of \( x \) and \( j \) over \( \mathbb{C} \) is 1, they  satisfy an algebraic relation \( F_N(x, j) = 0 \). Moreover, since  \( \varphi \) is a morphism defined over \( \mathbb{Q} \), \( x \) admits a Fourier expansion at the cusp \( \infty \) with rational coefficients. Consequently, we have $F_N(x, j) \in \mathbb{Z}[x, j].$  We define \( F_N(x, j) \) to be the minimal  polynomial of  \( x \) over
 \(\mathbb{C}(j) \) with coefficients $c_{k,l}\in \mathbb{Z}$ in $(1)$.

Let \(\{M_n \mid 1 \leq n \leq \mu\}\) be a complete set of right coset representatives for \(\Gamma_0(N)\) in \( SL_2(\mathbb{Z}) \). Denote by \( d = \deg(\varphi) \) the degree of the modular parametrization \( \varphi \).

\begin{proposition}
$K \leq \mu$ and $L \leq 2d$.
\end{proposition}

\begin{proof}
Given $\alpha \in \mathbb{C}$,  there exist generally  two distinct values $\beta_1, \beta_2 \in \mathbb{C}$ such that $(\alpha, \beta_1), (\alpha, \beta_2) \in E$. Consequently, there are typically $2d$ distinct points $[\tau_1], \cdots, [\tau_{2d}]$ on $X_0(N)$ satisfying $x([\tau_1]) = \cdots = x([\tau_{2d}]) = \alpha$.
This implies that the equation $F_N(\alpha, j) = 0$ has $2d$ distinct solutions at most, and therefore $L \leq 2d$. In other words, the degree of the field extension $\mathbb{C}(x, j)/\mathbb{C}(x)$ does not exceed $2d$.

For any $\beta \in \mathbb{C}$, there exists $\tau \in \mathcal{H}$ such that $\beta = j(\tau)$. For all $n$ with $1 \leq n \leq \mu$, we have $F_N(x(M_n \tau), j(M_n \tau)) = 0.$
Since $j(M_n \tau) = \beta$ for all $1 \leq n \leq \mu$, this shows that the equation $F_N(x, \beta) = 0$ generally has $\mu$ distinct solutions, and consequently $K \leq \mu$.  In other words, the degree of the field extension $\mathbb{C}(x, j)/\mathbb{C}(j)$ does not exceed $\mu$.
\end{proof}

Let \[x(\tau) = \sum_{n=-2}^\infty a_n q^n, \,\,\,
j(\tau) = \sum_{n=-1}^\infty b_n q^n, \,\,\,x^k j^l = \sum_{n=-2(k+l)}^\infty c(k, l; n) q^n,\]
where $q = {e^{2\pi i\tau }}$.
Then
\[
\sum_{k=0}^K \sum_{l=0}^L c_{k,l}\sum_{n=-2(k+l)}^\infty  c(k, l; n) q^n \equiv 0.
\]
Thus \[\sum\limits_{n =  - 2K - L}^\infty  {\left( {\sum\limits_{k = 0}^K {\sum\limits_{l = 0}^L {{c_{k,l}}} } c(k,l;n)} \right){q^n}}  \equiv 0. \]
This implies for all $n \geq -2K - L$,
\[
\sum_{k=0}^K \sum_{l=0}^L c_{k,l} c(k, l; n) = 0.
\]

Thus the coefficients $c_{k,l}$ can be determined
by solving this system of linear equations. Then we have the following algorithm.

\begin{algorithm}[H]
\caption{Construction of Modular Polynomial $F_N(x, j)$}
\label{alg:modular_poly}
\begin{algorithmic}[1]
\REQUIRE Elliptic curve $E$ over $\mathbb{Q}$ with conductor $N$.
\ENSURE Modular polynomial $F_N(x, j) = \sum_{k=0}^{K} \sum_{l=0}^{L} c_{k,l} x^k j^l$.
\STATE Compute the index $\mu$ of $\Gamma_0(N)$ in $SL_2(\mathbb{Z})$, set $K := \mu$;
 Compute the degree $d $ of  $\varphi$, set $L := 2d$;
 Set $M := (K+1)(L+1)$.
\STATE Compute the $q$-expansion of $x$ to precision $M$;
Compute the $q$-expansion of the modular invariant $j$ to precision $M+1$.

\STATE  For each pair $(k,l)$ with $0 \leq k \leq K$ and $0 \leq l \leq L$, compute the $q$-expansion of
\[
x^k j^l := \sum_{n=-2K-L}^{M} c(k,l;n) q^n.
\]
\STATE Solve the linear system in variables $\{c_{k,l}|0 \leq k \leq K, 0 \leq l \leq L\}$:
\[ \sum_{k=0}^{K} \sum_{l=0}^{L} c_{k,l} c(k, l; n) = 0 \quad \text{for } -2K-L \leq n \leq K(L-1). \]

\STATE Define and output $F_N(x, j) := \sum_{k=0}^{K} \sum_{l=0}^{L} c_{k,l} x^k j^l$.
\end{algorithmic}
\end{algorithm}

\begin{remark}
(i) In Step 1,  the degree $d $ of  $\varphi$ can be computed using the \texttt{ellmoddegree} function in PARI/GP;
(ii) In Step 2, the $q$-expansion of $x$ can be computed using the \texttt{elltaniyama} function in PARI/GP;
(iii) With minor modifications to Algorithm 1, one can compute the modular polynomials
$f_N(x, J)$, $  G_N(y, j) \, \text{and}\,\, g_N(y, J).$
\end{remark}

\begin{example}

The computational results for $N=11$ yield
\begin{align*}
F_{11}(x,j) &= (16 - x)^{11} j^2 \\
&+ (-104748564078368391 + 199736619430410535x + 159480622275659333x^2 \\
&\quad + 6839041777752481x^3 - 29669709666741936x^4 - 4074814667347831x^5 \\
&\quad + 1134855511654843x^6 + 164063633585170x^7 + 5072626276355x^8 \\
&\quad + 38323813979x^9 + 43119747x^{10} + 1486x^{11}) j \\
&+ (9789217 + 4971236x + 1333262x^2 - 52820x^3 + x^4)^3,
\end{align*}
\begin{align*}
f_{11}(x,J) &= (16 - x) J^2 + (6969 + 5732x - 12529x^2 - 6105x^3 + 1309x^4 + 297x^5 - 22x^6) J \\
&+ (97 + 116x + 62x^2 - 20x^3 + x^4)^3.
\end{align*}

\end{example}
\section{Algorithm for computing  the fiber of $\varphi$ over a point of $E$}
Let $[\tau] \in X_0(N)$,  $\alpha = x(\tau)$ and $\beta = j(\tau)$ such that  $\alpha, \beta \in \mathbb{C}$. By the definition of $F_N(X,Y)$, we have $F_N(\alpha, \beta) = 0$. Moreover, we have the following proposition.

\begin{proposition}
Let  $\alpha, \beta \in \mathbb{C}$ such that  $F_N(\alpha, \beta) = 0$.  Then there exists $\tau \in \mathcal{H}$ such that $\alpha = x(\tau)$ and $\beta = j(\tau)$.
\end{proposition}
\begin{proof}[Proof.]
Define an affine curve
$C_N := \{ (\alpha, \beta) \in \mathbb{C}^2 \mid F_N(\alpha, \beta) = 0 \}$,
and let $\overline{C}_N$ be the projective closure of $C_N$.
Then  there is a canonical desingularization  $g: Y \to \overline{C}_N$ over  $\mathbb{Q}$ of $\overline{C}_N$, i.e.,   a   nonsingular projective  curve $Y$ over $\mathbb{Q}$ and  a surjective proper morphism
$g: Y \to \overline{C}_N$ that is an isomorphism except over the singular points of $\overline{C}_N$, and the
pair $(Y, g: Y \to \overline{C}_N)$  is uniquely determined by $\overline{C}_N$ (up to unique isomorphism).

 The embedding $i: \mathbb{C}(x, j) \to \mathbb{C}(j, j_N)$ induces a surjective morphism
$h: X_0(N) \to Y.$ Therefore  $g\circ h: X_0(N) \to \overline{C}_N$ is surjective.  It follows from the  irreducibility of $F_N(X,Y)$ and the uniqueness of $Y$ that for any $\alpha, \beta \in \mathbb{C}$ with $F_N(\alpha, \beta) = 0$, there exists $\tau \in \mathcal{H}$ such that $\alpha = x(\tau)$ and $\beta =j(\tau)$.
\end{proof}
Given $P=(\alpha, \beta) \in E$, we could compute $j$ corresponding to $\alpha$ by solving  the equation $F_N(\alpha, j) = 0$,  which gives  the possible points  in  $\varphi^{-1}(P)$ in $\mathcal{H}$. Then we have  the following algorithm.
\begin{algorithm}[H]
\caption{Algorithm for computing  the fibers  of $\varphi$}
\begin{algorithmic}[1]
\REQUIRE  $P=(\alpha, \beta) \in E,F_N(x, j),f_N(x, J).$
\ENSURE  $\varphi^{-1}(P)$.
\STATE  Solve $F_N(\alpha, j) = 0$  for $j$,  let $j_1, j_2, \cdots, j_{L'}$   be its  roots,  here  $L' \leq L$.
 \STATE  Solve $f_N(\alpha, J) = 0$ for $J$,  let $Z$ be  the set of  its roots.
\STATE  For each { $1 \leq i \leq L'$},
 if {$j_i = 1728$} set $\tau_i := i$;
else if {$j_i = 0$} set  $\tau_i := \frac{-1 + i\sqrt{3}}{2}$;
else construct elliptic curve
    \[ E_i : y^2 + xy = x^3 - \frac{36}{j_i - 1728}x - \frac{1}{j_i - 1728} \]
and let $\omega_1, \omega_2$ be the periods of $E_i$,
    if {$\frac{\omega_1}{\omega_2} \in \mathcal{H}$}
        set $\tau_i := \frac{\omega_1}{\omega_2}$
    else set $\tau_i := \frac{\omega_2}{\omega_1}$.
\STATE Initialize  vector $\text{Inv}:= \emptyset$;
construct a complete system of right coset representatives $\{M_n : 1 \leq n \leq \mu\}$ for $\Gamma_0(N)$ in $SL_2(\mathbb{Z})$.
\STATE  For all { $1 \leq i \leq L'$}, { $1 \leq n \leq \mu$},   if $j(N\cdot M_n\tau_i) \in Z$ and  $\varphi(M_n\tau_i) = (\alpha, \beta)$, append $M_n\tau_i$ to $\text{Inv}$;
\STATE Compute values of $\varphi$ at all cusps using Algorithm 4;  if any cusp maps to $P$ under $\varphi$, append it to $\text{Inv}$.
\STATE  Return $\text{Inv}$.
\end{algorithmic}
\end{algorithm}

\begin{remark}
There exist $\tau_1, \tau_2 \in \mathcal{H}$ with $[\tau_1] \neq [\tau_2]$, $j(\tau_1) = j(\tau_2)$ and $j(N\tau_1) = j(N\tau_2)$. These points correspond to singular points of the plane model of $X_0(N)$ determined by the classical modular equation \cite{Klyachko-Kara,Wang}. For such points  $[\tau_1] $ and $ [\tau_2]$, we generally have
$\varphi([\tau_1]) \neq \varphi([\tau_2]).$
Therefore, in Step 3, points of $X_0(N)$ determined by $(j,J)$ may not belong to $\varphi^{-1}(P)$, and we must verify their images under $\varphi$ to make the final determination.
\end{remark}

\begin{example}
Consider the elliptic curve $E_{a1}(389)$ defined by:
\begin{equation}
y^2 + y = x^3 + x^2 - 2x.
 \end{equation}
The genus $g$ of $X_{0}(389)$  is 32,
$\mu = [SL_2(\mathbb{Z}):\Gamma_0(389)] = 390$,
 $\deg(\varphi) = 40$,
$K = \deg_{F_{389}}(x) = 390 = \mu$, $L = 2 \cdot \deg(\varphi) = 80$.
We could factor the polynomial $F_{389}(0,j)$ into two irreducible polynomials $f_1(j)$ and $f_2(j)$, where
\begin{align*}
f_1(j) &:= j^{40} + \cdots + 2^{60} \cdot 3^{12} \cdot 5^{12} \cdot 41^6 \cdot 59^3 \cdot 83^6 \\
&\quad \cdot 25284173^3 \cdot 2266783937522246669759^3 \\
&\quad \cdot 21020023204008338127670158568049294014 \\
&\quad \cdot 9512434767923557515056430506879436133^3 \\
f_2(j) &:= j^{40} + \cdots + 2^{59} \cdot 3^{12} \cdot 5^{11} \cdot 11^{12} \cdot 8429^3 \cdot 170801^3 \\
&\quad \cdot 99507788413092762738223^3 \\
&\quad \cdot 7060906424360528718534826263851663910667 \\
&\quad \cdot 9308687516043502683423204467607^3
\end{align*}

For a specific point $\tau \approx -0.02127691418 + 2.268594296 \times 10^{-5}i$, we have:

 $\varphi(\tau) = (0,0)$,
 $j(\tau) \approx -8.747731379 \times 10^{53} - 2.618292441 \times 10^{54}i$
 and $j(389\tau) \approx -458.3909478 - 562.7806563i$,
where $j(\tau)$ and $j(389\tau)$ are roots of $f_2(j)$ and $f_1(j)$ respectively.

 The rank of $E(\mathbb{Q})$ is 2, and $(0,0)$ is one of its generator of infinite order. Although points $\tau \in \varphi^{-1}((0,0))$ are transcendental complex numbers in $\mathcal{H}$, we have $(j(\tau), j(389\tau)) \in \overline{\mathbb{Q}}^2$. In general, when $(j(\tau), j(N\tau)) \in \overline{\mathbb{Q}}^2$, we can determine the minimal polynomials of $j(\tau)$ and $j(N\tau)$ over $\mathbb{Q}$.  For a given $\tau$, we can construct the point:
\[
P := \sum_{\sigma \in G} \varphi(\tau)^\sigma
\]
where $G$ is the Galois group of the splitting field of $\mathbb{Q}(j(\tau), j(N\tau))$. Clearly, $P \in E(\mathbb{Q})$. This shows that even when $\tau$ is transcendental, if $j(\tau)$ is algebraic, then $\varphi(\tau)$ is an algebraic point on $E$. Thus we could compute its trace.
\end{example}
For an imaginary quadratic order  $\mathcal{O} = {\mathcal{O}_D}$ of discriminant $D < 0$,  let $H_D(x)$  denote its Hilbert class polynomial.  Below are the generators  $P$  of strong Weil curves  $E$ of rank 1 with conductors not exceeding 100, along with the polynomials  for
$\prod\limits_{[\tau]  \in {\varphi ^{ - 1}}(P)} {(x - j(\tau ))} $.  This table can be regarded as a supplement to Cremona's Table 2.

\begin{center}
\begin{tabular}{lll}
\toprule
$E$ & Generator $P$ of $E(\mathbb{Q})$ &$\prod\limits_{\tau  \in {\varphi ^{ - 1}}(P)} {(x - j(\tau ))} $ \\
\midrule
37a1 & $(0,0)$ & $H_{-7}(x)^2$ \\
43a1 & $(0,0)$ & $H_{-7}(x)^2$ \\
57a1 & $(2,1)$ & $H_{-48}(x)^2$ \\
58a1 & $(0,1)$ & $H_{-20}(x)^2$ \\
61a1 & $(1,0)$ & $H_{-12}(x)^2$ \\
65a1 & $(1,0)$ & $H_{-16}(x)^2$ \\
77a1 & $(2,3)$ & $H_{-35}(x)^2$ \\
79a1 & $(0,0)$ & $H_{-7}(x)^2$ \\
82a1 & $(0,0)$ & $H_{-16}(x)^2$ \\
83a1 & $(0,0)$ & $H_{-19}(x)^2$ \\
88a1 & $(2,2)$ & $H_{-96}(x)^2$ \\
89a1 & $(0,0)$ & $H_{-8}(x)^2$ \\
91a1 & $(0,0)$ & $H_{-27}(x)^4$ \\
91b1 & $(-1,3)$ & $f_{9101}(x)$ \\
92b1 & $(1,1)$ & $H_{-44}(x)^2$ \\
99a1 & $(0,0)$ & $H_{-8}(x)^4$ \\
\bottomrule
\end{tabular}
\end{center}
The polynomial for  $E_{b1}(91)$ is defined as
\begin{align*}
f_{91b1}(j) &:= 2^{15} \cdot 5^{11} \cdot 991^3 \cdot 16572269011^3 \\
&\quad - 2^9 \cdot 3^4 \cdot 5^{10} \cdot 23 \cdot 40341091849 \cdot 788088043784206924489867 \, j \\
&\quad + 5^6 \cdot 248698909 \cdot 269818358221989089513089057757 \, j^2 \\
&\quad - 2^7 \cdot 3^4 \cdot 5^4 \cdot 49633 \cdot 5413524016643 \, j^3 \\
&\quad + 1048576 \, j^4.
\end{align*}

\begin{remark}
The elliptic curve $E_{b1}(91)$  is defined by
\[
y^2 + y = x^3 + x^2 - 7x + 5.
\]
and  $(-1,3)$ is its generator of infinite order.
Let $\varphi^{-1}((-1,3)) = \{[\tau_1], [\tau_2], [\tau_3], [\tau_4]\}$. Then the minimal polynomial of each $j(\tau_i)$ is $f_{91b1}(j)$. Since the leading coefficient of $f_{91b1}(j)$ is not 1, all $j(\tau_i)'s$ are not algebraic integers. Thus all  $\tau_i 's$ are transcendental numbers.
In contrast, for all other elliptic curves in the table, the generators of infinite order correspond to  CM-points  on $X_0(N)$.
\end{remark}

\section{Algorithm for Computing Poles of $\varphi$}

For any $\tau \in \mathcal{H}$, we have
\[
x(\tau) = \infty \iff \wp(\gamma(\tau)) = \infty \iff \gamma(\tau) \in L,
\]
where $L$ is the period lattice of $E$. Therefore, the set $x^{-1}(\infty) = \{\tau \in \mathcal{H}^{*} : x(\tau) = \infty\}$ contains zeros of $\gamma(\tau)$.

Let $x(\tau_0) = \infty$ for some $\tau_0 \in \mathcal{H}$. Then there exists a punctured neighborhood $U$ of $\tau_0$ such that for all $\tau \in U$, $x(\tau ) \ne \infty $ and
\[
F_N(x(\tau), j(\tau)) = \sum_{k=0}^K A_k(j(\tau))x(\tau)^k  \equiv 0.
\]
This implies
\[
A_K(j) + \frac{A_{K-1}(j)}{x} + \cdots + \frac{A_0(j)}{x^K} = 0.
\]
Taking the limit as $\tau \to \tau_0$ (and thus $x \to \infty$) yields $A_K(j(\tau_0)) = 0$, showing that
$j(\tau_0)$ is a zero of $A_K(j) = 0$. Then we have the following algorithm.

\begin{algorithm}[H]
\caption{Computation of Non-Cuspidal Poles of $\varphi$}
\begin{algorithmic}[1]
\REQUIRE Modular polynomial $F_N(x,j)$.
\ENSURE Non-cuspidal poles of $\varphi$.
\STATE Construct right coset representatives $\{M_n | 1 \leq n \leq \mu\}$ of $\Gamma_0(N)$ in $SL_2(\mathbb{Z})$ and set  $S:=\emptyset$.
\STATE Find zeros $j_1, \ldots, j_{L'}$ of $A_K(j) = 0$.
\STATE  For each $j_k$,  find $z_k$ such that $j(z_k) = j_k$ using  the algorithm in Step 3 of Algorithm 2.
\STATE    For each  $1 \leq k \leq L'$, $1 \leq n \leq \mu$,
          if $x(M_n z_k) = \infty$,  append $M_n z_k$ to $S$.
\STATE   Output $S$.
\end{algorithmic}
\end{algorithm}

\begin{remark}
Given $\beta \in \mathbb{C}$, define
\[
Z(\beta) = \{x \in \mathbb{C} | F_N(x,\beta) = 0\}, \quad
P(\beta) = \{[\tau] \in X_0(N) | x(\tau) = \infty, j(\tau) = \beta\}.
\]
Then $K = \# I(\beta) + \# Z(\beta)$, where both $Z(\beta)$ and $I(\beta)$ are treated as multisets.
\end{remark}

\begin{example}

Consider the elliptic curve $E_{a1}(38)$ defined by:
\[
y^2 + xy + y = x^3 + 9x + 90.
\]
The genus $g$ of $X_{0}(38)$  is 4,
$\mu = [SL_2(\mathbb{Z}):\Gamma_0(38)] = 60$,
 $\deg(\varphi) = 6$,
 $L = 2\deg(\varphi) = 12$, $K = \mu = 60$
  and its fundamental periods  of $E_{a1}(38)$ are
  \begin{align*}
    \omega_1 &\approx 1.8906322299422985362\\
    \omega_2 &\approx 0.9453161149711 - 0.601313878981i.
\end{align*}

The leading coefficient of $F_{60}(x,j)$ as a polynomial in $x$ is

$
A_{60}(j) = (256453780788797510341879944067 - 5295752119436627969393180j
+ \\
39952945838749259709j^2 - 4534998854j^3 + j^4)^2
$

The equation $A_{60}(j) = 0$ has four solutions:
\begin{align*}
j_1 &= \frac{2267499427}{2} - 1703222995i\sqrt{3} - \frac{1}{2}\sqrt{A-Bi\sqrt{3}}, \\
j_2 &= \frac{2267499427}{2} - 1703222995i\sqrt{3} + \frac{1}{2}\sqrt{A-Bi\sqrt{3}}, \\
j_3 &= \frac{2267499427}{2} + 1703222995i\sqrt{3} - \frac{1}{2}\sqrt{A+Bi\sqrt{3}}, \\
j_4 &= \frac{2267499427}{2} + 1703222995i\sqrt{3} + \frac{1}{2}\sqrt{A+Bi\sqrt{3}},
\end{align*}
where
\[
A = -29669607874801294131 \quad \text{and} \quad B = 15447089883389873728.
\]

These solutions correspond to points on the modular curve $X_0(38)$ as follows.
\begin{align*}
\tau_1 &\approx -0.03335482984 + 0.0006179683i, \\
\tau_2 &\approx 0.45453918428 + 0.0023009932i, \\
\tau_3 &\approx -0.12394615641 + 0.008571122i,\\
\tau_4 &\approx 0.48150957905 + 0.0003798451i.
\end{align*}

Direct computation yields
\[
\gamma(\tau_1) = \gamma(\tau_3) = 0, \quad \gamma(\tau_2) = \gamma(\tau_4) = -\omega_1.
\]

This shows that both $\tau_1$ and $\tau_3$ are zeros of the Eichler integrals. In \cite{Kodgis},
Kodgis conjectured that all zeros of the Eichler integrals $\gamma$  should be CM-points of $X_0(N)$.
However, since neither $j_1$ nor $j_3$ are singular moduli,
both $\tau_1$ and $\tau_3$ are transcendental numbers.
This provides a counterexample to Kodgis's conjecture, demonstrating that not all zeros of Eichler integrals are CM-points.
\end{example}

\section{Algorithm for Computing Exact Values of $\varphi$ at Cusps}

Let $[\tau_0]$ be a cusp of $X_0(N)$ with $x(\tau_0) \neq \infty$. Then
 $j(\tau_0) = \infty$ and
there exists a punctured neighborhood $U$ of $\tau_0$ such that for all $\tau \in U$,
  $j(\tau)\ne \infty$ and
    \[
    F_N(x(\tau), j(\tau)) = \sum_{l=0}^L B_l(x(\tau)) j(\tau)^l \equiv 0.
    \]
  This implies
    \[
    B_L(x) + \frac{B_{L-1}(x)}{j(\tau)} + \cdots + \frac{B_0(x)}{j(\tau)^L} \equiv 0.
    \]
Taking the limit as $\tau \to \tau_0$ (so $j(\tau) \to \infty$) yields
\[
B_L(x(\tau_0)) = 0.
\]
Thus $x(\tau_0)$ is a root of $B_L(x) = 0$.
Hence
\[
\{\varphi([\tau]) \mid \tau \in \mathbb{Q} \cup \{\infty\}\} = \{x \in \mathbb{C} \mid B_L(x) = 0\} \cup \{\infty\}.
\]
\begin{remark}
If the equation $B_L(x) = 0$ has no solutions, this indicates that all cusps of $X_0(N)$ map to $\infty$ under $\varphi$.
\end{remark}

By combining the solutions of $B_L(x) = 0$ with an approximation algorithm, we have the following algorithm to compute the exact values of $\varphi$ at cusps. For all $M \in {\Gamma _0}(N),$  define ${P_f}(M): = \gamma (M{\tau _0}) - \gamma ({\tau _0})$
and call it a period of $f$.

\begin{algorithm}[H]
\caption{Computation of the exact values of $\varphi$ at cusps}
\begin{algorithmic}[1]
\REQUIRE modular polynomial $F_N(x,j)$ and $G_N(y, j).$
\ENSURE Exact values of $\varphi$ at  cusps.

\STATE Initialize an empty vector $Z \gets \emptyset$;
for each root $x_i$ of $B_L(x) = 0$ and
 each root $y_j$ of $C(y) = 0$,
       if  $(x_i, y_j)$ satisfies equation (2)
 $Z \gets Z \cup \{(x_i, y_j)\}$,  where $C(y)$ is
 the leading coefficient of $G_N(y, j)$  as a polynomial in $j$.

\STATE For each cusp $[\frac{s}{r}]$ of $X_0(N)$,
compute the approximate value of $\varphi([\frac{s}{r}])$
by the following step.

{(2.1)}  Find minimal prime $p$ with $(p,N) = 1$ such that
   $[\frac{s}{r}] = [\frac{ps}{r}] = [\frac{rj + s}{pr}]$ for all $0 \leq j \leq p-1$,
 where we use the  criterion from Proposition 2.2.3  in \cite{Cremona}  to
decide whether or not two cusps are equivalent.

{(2.2)}  For each $j = 0,1,\cdots,p$, construct $M_j \in \Gamma_0(N)$ such that
\[ M_p\left(\frac{ps}{r}\right) = M_j\left(\frac{rj + s}{pr}\right) = \frac{s}{r} \]
by the method contained in the proof of Proposition 2.2.3 in \cite{Cremona}.

{(2.3)}  For $j = 0,1,\cdots,p$,
compute the period $P_f(M_j)$ using the method from Proposition 2.10.3 in  \cite{Cremona}.

{(2.4)}  Compute
\[ \gamma\left(\frac{s}{r}\right) := \frac{1}{p+1-a_p} \sum_{j=0}^p P_f(M_j) \]
to a given precision,
where $a_p$ is the coefficient of $q^p$ in the $q$-expansion of $f$.  Recall that
 $f$ is the newform for ${\Gamma _0}(N)$  associated to $E$ by $\varphi$.

{(2.5)}  Calculate
$ \varphi\left(\left[\frac{s}{r}\right]\right) $ to  a given precision using $(4), (5)$ and  the value of
$\gamma\left(\frac{s}{r}\right)$ in (2.4),
and compare the approximate value  with the values in $Z$ to determine the exact value of $\varphi$ at
$\left[\frac{s}{r}\right]$.
\STATE  Output results from Steps 2.
\end{algorithmic}
\end{algorithm}

\begin{remark}
The  approximation algorithm in Step 2  is  a minor improvement to Delaunay's algorithm presented in \cite{Delaunay}.
\end{remark}

\begin{example}
Consider $N = 176$. The elliptic curve $E_{a1}(176)$ is given by
\[ y^2 = x^3 - 4x - 4. \]
The leading coefficient $B_L(x)$  in $F_{176}(x,j)$  as a polynomial in $j$ is 1.
Since $B_L(x) = 0$ has no solutions, this shows that $\varphi$ maps all cusps to $\infty$.
The leading coefficient $A_K(j)$ of $x$  in  $F_{176}(x,j)$  as a polynomial in $x$  is 1. Consequently, the equation $A_K(j) = 0$ has no solutions, which indicates that no points in $\mathcal{H}$ map to $\infty$ under $\varphi$.

Thus, the fiber over infinity consists entirely of cusps, i.e.,
\[
\varphi^{-1}(\infty) =\{[1], [\frac{1}{2}], [\frac{1}{4}], [\frac{3}{4}], [\frac{1}{8}], [\frac{1}{11}], [\frac{1}{16}], [\frac{1}{22}], [\frac{1}{44}], [\frac{3}{44}], [\frac{1}{88}], [\frac{1}{176}] \}.
\]

The modular parametrization $\varphi: X_0(176) \to E_{a1}(176)$ has degree 16.
However, $\varphi^{-1}(\infty)$ contains only 12 distinct points.
 Using the criterion of Corbett and Saha (Theorem 1.1, \cite{Corbett-Saha}), we find that $\varphi$ is ramified at the cusps $\left[\frac{1}{4}\right], \left[\frac{3}{4}\right], \left[\frac{1}{44}\right], \left[\frac{3}{44}\right]$ and
each has ramification index 2. Thus when counting with multiplicity, the fiber satisfies:
\[
\sum_{[\tau] \in \varphi^{-1}(\infty)} e_\varphi([\tau]) = 16,
\]
where $e_\varphi([\tau])$ denotes the ramification index at $[\tau]$.
\end{example}
\section{Algorithm for Determining Ramification Points of $\varphi$}

  Let $\tau_0 \in \mathcal{H}$ such that $[\tau_0]$ is a critical point of $\varphi: X_0(N) \to E$.
Let $\varphi([\tau_0]) = z_0\ne \infty$. Then  $z_0$ is a ramification  point of  $\varphi$. Let $e$ be the ramification index
of  $\varphi$ at $z_0$.
By the local mapping theorem,  for sufficiently small $\varepsilon > 0$, there exists $\delta > 0$ such that
for all $z$ in $D_2 = \{z \in \mathbb{C}: 0 < |z - z_0| < \delta\}$,
the equation $\varphi(\tau) = z$ has exactly $e$ solutions in $D_1 = \{\tau \in \mathcal{H}: 0 < |\tau - \tau_0| < \varepsilon\}$.
Moreover, we assume that  $D_2$ contains no ramification points.
For each $z \in D_2$,  let $\varphi^{-1}(z) = \{[\tau_1], \cdots, [\tau_d]\}$.
By suitably choosing  representatives $\tau_i$  in $[\tau_i]$ and $\varepsilon$,
we make  $\tau_i$ to an analytical function in $z$ for all $1 \leqslant i \leqslant d$.
Moreover, we  require that all  $\tau_1, \cdots, \tau_e$  in $D_1$.
Let $\alpha = x(\tau_0)$.
 As $z \to z_0$, we have $\tau_i \to \tau_0$ for $1 \leq i \leq e$.
 Thus $j(\tau_i) \to j(\tau_0)$ for $1 \leq i \leq e$.
 This implies the equation in $j$
\[ F_N(\alpha, j) = B_L(\alpha) \prod_{l=1}^L (j - j(\tau_l)) = 0 \]
has multiple roots.
Therefore, $\alpha$  is a common solution of
\[ F_N(x,j) = 0 \quad \text{and} \quad \frac{\partial F_N}{\partial j}(x,j) = 0. \]
This gives a necessary condition for $(\alpha ,\beta ) \in E$  to be a ramification point of $\varphi$.
Thus we have the following algorithm.
\begin{algorithm}[H]
\caption{Algorithm for Finding Ramification Points}
\begin{algorithmic}[1]
\REQUIRE
 $F_N(x,j)$, $f_N(x,J)$,  $G_N(y,j)$ and $g_N(y,J)$.

\ENSURE  Set of ramification points of $\varphi$.

\STATE  Compute
$R_1(x) =\text{Resultant}(F_N(x,j), \frac{\partial F_N}{\partial j}(x,j), j)$  and
$R_2(x) = \text{Resultant}(f_N(x,J), \frac{\partial f_N}{\partial J}(x,J), J) $ then
factor them and find their common factor $U(x)$.
\STATE  Compute
$ R_3(y) = \text{Resultant}(G_N(y,j), \frac{\partial G_N}{\partial j}(y,j), j) $  and
$R_4(y) = \text{Resultant}(g_N(y,J), \frac{\partial g_N}{\partial J}(y,J), J) $ then
factor  them and find their common factor $V(y)$.
\STATE  Find  all  $(x_i,y_j)$  with $(x_i,y_j) \in E$
for each solution $x_i$  of $U(x)=0$ and  $y_j$ of  $V(y)=0$;
then  compute $\varphi^{-1}((x_i,y_j))$ using Algorithm 2
and decide whether  $(x_i,y_j)$ is a ramification point.
 This  find the set of the ramification points in $Y_0(N)$.
\STATE
Using Corbett-Saha criterion (Theorem 1.1, \cite{Corbett-Saha}),
Brunault criterion (Theorem 1.2, \cite{Brunault}) and
Algorithm 4,  find the set of the ramification  points in cusps of $X_0(N)$.

\STATE  Output results from Steps 3 and 4.
\end{algorithmic}
\end{algorithm}

\begin{example}
Consider the elliptic curve $E_{a1}(46)$ defined by
\[ y^2 + xy = x^3 - x^2 - 10x + 12. \]
The genus $g$ of $X_{0}(46)$  is 5,
$\mu = [SL_2(\mathbb{Z}):\Gamma_0(46)] = 72$,
 $\deg(\varphi) = 5$,
$K = \deg_{F_{46}}(x) = 72 = \mu$, $L = 2 \cdot \deg(\varphi) = 10$.

{\bf  Step 1.}  The common factor of
$ \text{Resultant}\left(F_{46}(x,j), \frac{\partial F_{46}}{\partial j}(x,j), j\right) $
and \\
$
\text{Resultant}(f_{46}(x,J), \frac{\partial f_{46}}{\partial J}(x,J), J)
$  factors  as
$U = U_1 \cdot U_2 \cdots U_{14},$
where
\begin{align*}
U_1 &= -4 + x, \\
U_2 &= 4 + 3x + x^2, \\
U_3 &= 3184 + 2472x + 567x^2 - 70x^3 + 23x^4.
\end{align*}

{\bf  Step 2.}
The common factor of
$\text{Resultant}(G_{46}(y,j), \frac{\partial G_{46}}{\partial j}(y,j), j)$ and

$\text{Resultant}(g_{46}(y,J), \frac{\partial g_{46}}{\partial J}(y,J), J)$
factors as
$V = V_1 \cdot V_2 \cdot V_3$,  where
\begin{align*}
V_1 &= 2 + y, \\
V_2 &= 1996 - 1004y + 711y^2 + 14y^3 + 27y^4, \\
V_3 &=
7416293824 + 9784853696y - 1586824496y^2 \\
&- 3448946432y^3 + 971389940y^4 + 12153116y^5\\
&+ 7085747y^6 + 37030y^7 + 12167y^8.
\end{align*}

{\bf  Step 3.} We consider the following substeps:

{\bf Substep 3.1} The equation $V_1=0$  in $y$   corresponds to  $U_1 \cdot U_2 = 0$  in $x$.
These  corresponds to three points   $(4,-2)$, $(x_1,-2)$, $(x_2,-2)$  on $E$, where
$x_1 = \tfrac{1}{2}(-3 - i\sqrt{7}), \quad x_2 = \tfrac{1}{2}(-3 + i\sqrt{7})$.
None of  them  are ramification points of $\varphi$.  Here
the fiber $\varphi^{-1}((4,-2))$ contains the cusps $[1]$
and  $[\tfrac{1}{23}]$.
Define matrices and points $\tau_k = M_k \tau_0 \quad \text{for} $
\[
\quad M_1 = \begin{pmatrix} 0 & -1 \\ 1 & 0 \end{pmatrix}, M_2 = \begin{pmatrix} 0 & -1 \\ 1 & 5 \end{pmatrix},M_3 = \begin{pmatrix} 1 &14 \\ 2 & 29 \end{pmatrix}, M_4 = \begin{pmatrix} 1 & 16 \\ 2 & 33 \end{pmatrix},
\]
where $\tau_0 = \tfrac{7 + \sqrt{-7}}{4}$. These represent four distinct points on $X_0(46)$ but
\[
j(\tau_k) = J(\tau_k) = -3375 \quad (1 \leq k \leq 4).
\]
Thus they are singular points of $Z_0(46)$ (See Section 10). By computing the values of $\varphi$ at these points, we have
 $[\tau_1] \in \varphi^{-1}((x_1,-2))$,  $[\tau_2] \in \varphi^{-1}((x_1,y_1))$,  $[\tau_3] \in \varphi^{-1}((x_2,y_2))$ and  $[\tau_4] \in \varphi^{-1}((x_2,-2))$,
where $
y_1 = \frac{7 + i\sqrt{7}}{2}, \quad y_2 = \frac{7 - i\sqrt{7}}{2}.
$

The example shows that although $j = -3375$ is a double root of
$F_{46}(x_1,j)$,  there exist two distinct points $[\tau_1]$ and $[\tau_2]$ on
$X_0(46)$ satisfying  $x(\tau_1) = x(\tau_2) = x_1 \,\,\, \text{and} \,\,\,  j(\tau_1) = j(\tau_2) = -3375
$.

{\bf  Substep 3.2}
The equation $V_2=0$  in $y$   corresponds to
$(1873 + 160x - 234x^2 - 16x^3 + 9x^4)(112 + 40x - 57x^2 - 10x^3 + 9x^4) = 0,$
which doesn't  appear in the factorization of $U$. Then all such $(x_i,y_j)$ are not
ramification points.

{\bf  Substep 3.3}
The equation $V_3=0$  in $y$   corresponds to  $U_3 = 0$  in $x$.
 Computational results show
that the 8 points generated by these pairings are all ramification points of $\varphi$,
and  according to Brunault's criterion, none of  values of $\varphi$  at cusps are ramification points.
\end{example}
\section{Half-Trace of  CM-points}

Let \( x = (\phi : E \rightarrow E') \) be a CM-point on \( X_0(N) \).  Below, we illustrate through an example how to use \( F_N(x, j) \) to compute the half-trace of CM-points. In general, for CM-points, it is difficult  to establish an Artin isomorphism between the ideal class group of the imaginary quadratic field \( K = \mathbb{Q}(\sqrt{D}) \) and the Galois group \( G \) of the splitting field  of  \( \mathbb{Q}(j(E), j(E'), \sqrt{D}) \).  However, for certain special cases, we can still use ideal classes to compute \( G \), as shown in the following example.

\begin{example}
Let \( N = 389 \). Define \( \tau_0 = \frac{-1 + \sqrt{-3}}{2} \), \( \tau_1 = N \tau_0 \), and
\[
\Lambda = \mathbb{Z} + \mathbb{Z} \cdot \tau_0, \quad \Lambda' = \mathbb{Z} \cdot \frac{1}{389} + \mathbb{Z} \cdot \tau_0, \quad x = (\text{id} : \mathbb{C} / \Lambda \rightarrow \mathbb{C} / \Lambda').
\]
Then \( x \) is a CM-point on \( X_0(N) \). The complex form of \( x \) is \( [\tau_0] \). Our computations yield  that \( j(\tau_0) = 0 \)  and  \( j(\tau_1) \) is an algebraic integer of degree 130.

Let \( K = \mathbb{Q}(\sqrt{-3}) \), \( L = \mathbb{Q}(\sqrt{-3}, j(\tau_1)) \), and \( G \) be the Galois group of \( L \).
Let \( P = \varphi([\tau_0]) = (x_1, y_1) \in E \), where
\[
(x_1, y_1) \approx (52794.73595419, 12130816.356885).
\]
Then \( x_1 \) and \( y_1 \) are algebraic numbers of degree 130.

Let \( H \) be the subgroup of order 65 of \( G \). Then
\[
\sum_{\sigma \in H} P^\sigma = (x_H, y_H),
\]
where
\[
x_H = \frac{83054760289201}{68231958132816},
\]
\[
y_H = \frac{-5558102574231821778144 + 591325257523727551823\sqrt{389}}{11116205148463643556288}.
\]

Below are the main computational steps for the result above.

{\bf Step 1}
Substituting  \( j(\tau_0) = 0 \) into \( F_N(x, j(\tau_0)) \)  and factoring
the resulting polynomial yields  the minimal polynomial \( f_0(X) \) of degree 130
for the algebraic number  \( x_1 \).
Let \( a \) be the leading coefficient of \( f_0(X) \), and define \( f_1(X) := a^{129} f_0\left(\frac{X}{a}\right) \). Then \( f_1(X) \) is the minimal polynomial of \( a x_1 \) with the leading coefficient 1.

By computing the resultant of \( f_0(x) \) and  the polynomial  in $(6)$ with respect to the variable \( x \), we obtain a polynomial in \( y \). Factoring this polynomial yields two polynomials of degree 130. Using the approximate value of \( y_1 \) above, we can determine the minimal polynomial of the algebraic number \( y_1\), denoted as \( g_0(X) \). Similarly, let \( b \) be the leading coefficient of \( g_0(X) \),  and define \( g_1(X) := b^{129} g_0\left(\frac{X}{b}\right) \). Then \( g_1(X) \) is the minimal polynomial of \( b y_1 \) with the leading coefficient 1.
Then  we obtain the set of 130 points \( \{(x_n, y_n) : 1 \leq n \leq 130\} \) on \( E \) which are conjugates  of  \( (x_1, y_1) \).

{\bf Step 2}  Compute \( G_0 = \text{Gal}(L/K) \) and its multiplication table. For any \( z \in \mathcal{H} \), let \( \Lambda_z \) denote the lattice \( \mathbb{Z} + \mathbb{Z} \cdot z \). Let \( \mathcal{O} = \mathbb{Z} + \mathbb{Z} \cdot 389\tau_0 \), and let \( \{\Lambda_{\tau_1}, \dots, \Lambda_{\tau_{130}}\} \) be a representative system of the proper \( \mathcal{O} \)-ideal classes, where \( \tau_1 = 389\tau_0 \).
 Let \( [\Lambda_{\tau_i}] \) denote the ideal class of \( \Lambda_{\tau_i} \). By the Artin isomorphism theorem, we can identify \( G_0 \) with $\text{Pic}(\mathcal{O}) = \{ [\Lambda_{\tau_1}], \cdots, [\Lambda_{\tau_{130}}] \}$. The group operation on $G_0$ is defined by
\[
[\Lambda_{\tau_i}] \cdot [\Lambda_{\tau_j}] = [\Lambda_{\tau_k}] \quad \text{if} \quad [\Lambda_{\tau_i}\cdot\Lambda_{\tau_j}] = [\Lambda_{\tau_k}].
\]
Since $F_N(x,j)$ is non-degenerate in degrees, we have $$L = \mathbb{Q}(\sqrt{-3}, x_1) = \mathbb{Q}(x_1 - \sqrt{-3}).$$ Consequently, there exists a bijection between $\{ j(\tau_1), \cdots, j(\tau_{130}) \}$ and $\{ x_1, \cdots, x_{130} \}$.
Without loss of generality, we may pair $j(\tau_k)$ with $x_k$ such that for all $1 \leq k \leq 130$,
\[
F_{389}(x_k, j(\tau_k)) = 0.
\]
Under the Artin isomorphism, let $[\Lambda_{\tau_i}]$ correspond to $\sigma_i \in G_0$. Then:
\[
[\Lambda_{\tau_i}] \cdot [\Lambda_{\tau_j}] = [\Lambda_{\tau_k}] \quad \text{iff} \quad \sigma_i \circ \sigma_j = \sigma_k \quad \text{iff} \quad \sigma_j(j(\tau_k)) = j(\tau_i)
\]
iff $\sigma_j(x_k) = x_i$.

{\bf Step 3}  We compute the standard representation of the point set $\{(x_n, y_n)\}$ on $E$ in $L = \mathbb{Q}(x_1 - \sqrt{-3})$. Let $h(X)$ be the minimal polynomial of $x_1 - \sqrt{-3}$. Then all roots of $h(X)$ are $\{ x_k \pm \sqrt{-3} : 1 \leq k \leq 130 \}$. Define $X_k = ax_n$ and $Y_k = by_n$, where $X_k$ and $Y_k$ are algebraic integers in $L$.

For any $1 \leq i \leq 130$, there exists $1 \leq k \leq 130$ such that for all $0 \leq h \leq 259$,
\[
\sum_{j=1}^{130} \sigma_j\left(X_i (X_1 - \sqrt{-3})^{h} + X_i (X_k + \sqrt{-3})^{h}\right)\\
=
\]\[
\text{Tr}_{L/K}\left(X_i (X_1 - \sqrt{-3})^{h} + X_i (X_k + \sqrt{-3})^{h}\right) \in \mathbb{Z},
\]
 where $k$ is determined by the relation ${\sigma _1}{\sigma _i}^2 = {\sigma _k}$ in our configuration.
This allows us to use the Berg-Olivier algorithm \cite{Cohen}  to obtain a set of rational numbers
$\{ a_{i,h} : 1 \leq i \leq 130, 0 \leq h \leq 259 \}$ such that
\[
X_i = \sum_{h=0}^{259} a_{i,h} (X_1 - \sqrt{-3})^{h}.
\]

Similarly, we can obtain  $\{b_{i,h}:1 \leq i \leq 130, 0 \leq h \leq 259\}$ such that
\[
Y_i = \sum_{h=0}^{259} b_{i,h}(X_1 - \sqrt{-3})^h.
\]
This yields the standard representation of algebraic number pairs
\[
(x_i, y_i) = \left(\frac{X_i}{a}, \frac{Y_i}{b}\right).
\]

{\bf Step 4} The sum $\sum_{\sigma \in H} P^\sigma$ can be computed as follows.
 From Step 2, we obtain the subset of points $(x_i, y_i)$ on $E$ corresponding to the subgroup $H$.
 The exact representation of $x_i$ obtained in Step 3 can be expressed as \texttt{Mod}$(x_i, h(x))$ in the GP/PARI system, allowing standard arithmetic operations.
However, the addition function \texttt{elladd} of points on an  elliptic curve  in GP only accepts numerical inputs, then we have to give it a separate implementation.
For the values $x_n$  obtained in Step 3, each requires approximately 2GB of storage due to their large size.
Division operations are computationally inefficient with \texttt{Mod} representations. Our implementation of \texttt{elladd} uses only addition and multiplication, with final results requiring minor conversion.

\begin{remark}
In the above example, if the rationality of $x_H$ is known in advance,
one can compute numerical approximations of $x_H$ to sufficiently high precision
and then recover its exact value using the \texttt{bestappr} function in PARI/GP.
\end{remark}
\end{example}
\section{Some properties of $F_N(x,j)$ and $f_N(x, J)$}
We require that GCD($\{c_{k,l}\}$)= 1 and $a_{K,L} \geq 1$ in $(1) $ in order to ensure the uniqueness of $F_N(x, j)$.
The same restrictions apply to the polynomials $f_N(x, J)$, $G_N(y, j)$, and $g_N(y, J)$. Let $L(E, s)$ denote the $L$-series of $E$. This allows us to state the following proposition.

\begin{proposition}
If $L(E, 1) = 0$, then $F_N(X, Y) = f_N(X, Y)$.
\end{proposition}

\begin{proof}[Proof.]
Let $f|W_N = \varepsilon f$ where $\varepsilon = \pm 1$. Then we have
\[
f\left(\frac{-1}{N\tau}\right) = \varepsilon N\tau^2 f(\tau).
\]
Thus
\[
\gamma\left(\frac{-1}{N\tau}\right) = \int_{i\infty}^{\frac{-1}{N\tau}} 2\pi i f(z) dz = \int_0^\tau 2\pi i f\left(\frac{-1}{Nu}\right) \frac{1}{Nu^2} du
\]
\[
= \varepsilon \int_{i\infty}^\tau 2\pi i f(u) du - \varepsilon \int_{i\infty}^0 2\pi i f(u) du = \varepsilon \int_{i\infty}^\tau 2\pi i f(u) du
\]
since $\int_{i\infty}^0 2\pi i f(u) du = L(E, 1) = 0$.

Consequently, we obtain
\[
x\left(\frac{-1}{N\tau}\right) = \wp \circ \gamma\left(\frac{-1}{N\tau}\right) - \frac{b_2}{12} = \wp\left(\varepsilon \int_{i\infty}^\tau 2\pi i f(u) du\right) - \frac{b_2}{12} = x(\tau).
\]

The equality $F_N(X, Y) = f_N(X, Y)$ follows from
\[
j\left(\frac{-1}{N\tau}\right) = J(\tau) \quad \text{and} \quad x\left(\frac{-1}{N\tau}\right) = x(\tau).
\]
\end{proof}

Let $f(\tau) = \sum_{n=1}^{\infty} a_n q^n$.

\begin{lemma}
Let $N = 4m$ and suppose $\gamma\left(\frac{2}{N}\right) \in L$. Then
\begin{equation}
x\left(\frac{\tau}{2m\tau + 1}\right) = x(\tau).\end{equation}
\end{lemma}

\begin{proof}[Proof.]
Since $2^2 \mid N$, we have $a_2 = 0$  by $(f|U_2)(\tau) =
a_2\cdot f(\tau)$. Consequently,  $(f|U_2)(\tau) = 0$, which implies that for all $\tau \in \mathbb{H}$,
\[ f\left(\frac{\tau}{2}\right) + f\left(\frac{\tau + 1}{2}\right) = 0. \]
It follows that for all $\tau \in \mathbb{H}$,
\[ f(\tau) + f\left(\tau + \frac{1}{2}\right) = 0. \]

Now let $f|W_N = \epsilon f$ where $\epsilon = \pm 1$, yielding the transformation law:
\[ f\left(\frac{-1}{m\tau}\right) = \epsilon m\tau^2 f(\tau). \]
By
\begin{align*}
f\left(\frac{\tau}{2m\tau+1}\right) &= f\left(\frac{-1}{m\left(\frac{-1}{2} + \frac{-1}{m\tau}\right)}\right) \\
&= \epsilon m\left(\frac{-1}{2} + \frac{-1}{m\tau}\right)^2 f\left(\frac{-1}{2} + \frac{-1}{m\tau}\right) \\
&= -\epsilon m\left(\frac{-1}{2} + \frac{-1}{m\tau}\right)^2 f\left(\frac{-1}{m\tau}\right) \\
&= -m\left(\frac{-1}{2} + \frac{-1}{m\tau}\right)^2 m\tau^2 f(\tau) \\
&= -(2m\tau+1)^2 f(\tau),
\end{align*}
we have
\[f|\left( {\begin{array}{*{20}{c}}
  1&0 \\
  {2m}&1
\end{array}} \right)(z) =  - f(z).\]
Thus
\[ \int_{i\infty}^{i} 2\pi i f(z)dz + \int_{i\infty}^{i} 2\pi i f\left(\begin{pmatrix} 1 & 0 \\ 2m & 1 \end{pmatrix}z\right) dz = 0. \]
Under the substitution $u = \frac{z}{2mz+1}$, the second integral transforms as
\begin{align*}
\int_{i\infty}^{i} 2\pi i f\left(\begin{pmatrix} 1 & 0 \\ 2m & 1 \end{pmatrix}z\right) dz
& =\int_{i\infty}^{\tau} 2\pi i \frac{1}{(2mz+1)^2} f\left(\frac{z}{2mz+1}\right) dz\\
&= \int_{\frac{1}{2m}}^{\frac{\tau}{2m\tau+1}} 2\pi i f(u) du \\
&= \int_{i\infty}^{\tau} 2\pi i f(u) du - \int_{i\infty}^{\frac{1}{2m}} 2\pi i f(u) du.
\end{align*}

Consequently, we obtain
\[
\int_{i\infty}^{\frac{\tau}{2m\tau+1}} 2\pi i f(u) du + \int_{i\infty}^{\tau} 2\pi i f(z) dz = \int_{i\infty}^{\frac{1}{2m}} 2\pi i f(u) du = \gamma\left(\frac{2}{N}\right).
\]

Since $\gamma\left(\frac{2}{N}\right) \in L$, it follows that:
\[
x\left(\frac{\tau}{2m\tau+1}\right) = x(\tau).
\]
\end{proof}

\begin{remark}
If $4 \mid N$, then the matrix
\[
w = \begin{pmatrix}
1 & 0 \\
\frac{N}{2} & 1
\end{pmatrix} \in \Gamma_0\left(\frac{N}{2}\right)
\]
 belongs to the normalizer of $\Gamma_0(N)$ in $\mathrm{SL}_2(\mathbb{R})$. Thus $w$ define an automorphism of $X_0(N)$. To our knowledge, this automorphism  has not been described in the literature \cite{Atkin-Lehner}. In fact,  $(7)$  implies  that $x$  is a function on $X_0(\frac{N}{2})$.
\end{remark}

Recall that $E \cong \mathbb{C}/L$ and $\{M_n \mid 1 \leq n \leq \mu\}$ is a complete set of right coset representatives for $\Gamma_0(N)$ in $\mathrm{SL}_2(\mathbb{Z})$.
\begin{proposition}
If $4 \mid N$ and $\gamma\left(\frac{2}{N}\right) \in \Lambda$, then the degrees of $F_N(x,j)$  satisfy
\[
K \leq \frac{\mu}{2} \quad \text{and} \quad L \leq d.
\]
\end{proposition}

\begin{proof}[Proof.]
For any given $\beta \in \mathbb{C}$, there exists $\tau \in \mathcal{H}$ such that $\beta = j(\tau)$.  For all $1 \leq n \leq \mu$, we have  $ F_N(x(M_n \tau), \beta) = 0. $

Since $\Gamma_0(N)M_n \neq \Gamma_0(N)\begin{pmatrix} 1 & 0 \\ \frac{N}{2} & 1 \end{pmatrix} M_n$,  there exists an index
$ 1 \leq l \leq \mu, \,\,\, l \neq n$
such that
\[ \Gamma_0(N)\begin{pmatrix} 1 & 0 \\ \frac{N}{2} & 1 \end{pmatrix} M_n = \Gamma_0(N)M_l. \]

From Lemma $8.2$,  it follows that $x(M_n \tau) = x(M_l \tau)$.  Consequently, the set $\{x(M_n \tau)\}_{1 \leq n \leq \mu}$ contains at most $\frac{\mu}{2}$ distinct elements. Therefore, the equation $F_N(x, \beta) = 0$ generally has $\frac{\mu}{2}$ distinct solutions, which implies  $ K \leq \frac{\mu}{2} $. Thus
 the degree of the field extension $\mathbb{C}(x, j)/\mathbb{C}(j)$ is bounded by $\frac{\mu}{2}$.

For any fixed $\alpha \in \mathbb{C}$, there generically exist $2d$ distinct points $[\tau_1], \cdots, [\tau_{2d}]$ on $X_0(N)$ satisfying
\[ x([\tau_1]) = \cdots = x([\tau_{2d}]) = \alpha. \]
For each $[\tau_i]$, our previous discussion shows there exists $[M_n \tau_i] \neq [\tau_i]$ with
\[ x([M_n \tau_i]) = x([\tau_i]) = \alpha. \]
Thus $[M_n \tau_i] \in \{[\tau_1], \cdots, [\tau_{2d}]\}$.
Since $j(M_n \tau_i) = j(\tau_i)$, the set $\{j(\tau_1), \dots, j(\tau_{2d})\}$ contains at most $d$ distinct values. Consequently, the equation $F_N(\alpha, j) = 0$ admits at most $d$ distinct solutions, yielding $L \leq d$. We therefore conclude that the degree of the field extension $\mathbb{C}(x, j)/\mathbb{C}(x)$ is bounded by $d$.
\end{proof}
\begin{example}
Consider the elliptic curve $E_{a1}(40)$ defined by:
\[ y^2 = x^3 - 7x - 6. \]
The genus $g$ of $X_{0}(40)$  is 3,
$\mu = [SL_2(\mathbb{Z}):\Gamma_0(40)] = 72$,
 $\deg(\varphi) =2$,
$K = \deg_{F_{40}}(x) = 36 =  \frac{\mu}{2}$, $L =  \deg(\varphi) =2$.
This shows that the degrees of $F_{40}(x, j)$  are  half of the generic case.
The integral points of $X_0(40)$ are $(-2, 0)$, $(-1, 0)$, and $(3, 0)$. We observe that the equation $F_{40}(-2, j) = 0$  has no solutions.
 Similarly, $F_{40}(-1, j) = 0$ and $F_{40}(3, j) = 0$ have no solutions.
In fact,
$\varphi^{-1}((-2, 0)) = \left(\frac{1}{4}, \frac{1}{8}\right),
\varphi^{-1}((3, 0)) = \left(1, \frac{1}{2}\right), $
$\varphi^{-1}((-1, 0))= \left(\frac{1}{5}, \frac{1}{10}\right),
\varphi^{-1}(\infty) = \left(\frac{1}{20}, \frac{1}{40}\right).
$
Take the points $P_1 = (-3, 2\sqrt{3}i)$ and $P_2 = (-3, -2\sqrt{3}i)$.
The equation
$ F_{40}(-3, j) = 1073741824 \cdot (1141195895676649024 +
 944485450025040j + 847288609443j^2) = 0
$
has two solutions
\begin{align*}
j_1 &= \frac{10648(-182511805 - 236108339i\sqrt{2})}{3486784401}, \\
j_2 &= \frac{10648(-182511805 + 236108339i\sqrt{2})}{3486784401}.
\end{align*}

For the approximate $\tau$-values
\begin{align*}
\tau_{11} &\approx -0.272844389235320 + 0.00701737923645263i, \\
\tau_{12} &\approx 0.117863782898542 + 0.00292282078245356i, \\
\tau_{21} &\approx -0.25000000000000 + 0.0813850458467884i, \\
\tau_{22} &\approx 0.134961072356135 + 0.0122394381578248i,
\end{align*}
we have the $j$-invariant relations
\[ j(\tau_{11}) = j(\tau_{12}) = j_1, \quad j(\tau_{21}) = j(\tau_{22}) = j_2, \]
and
\[ \varphi^{-1}(P_1) = (\lceil \tau_{12} \rceil, \lceil \tau_{21} \rceil), \quad \varphi^{-1}(P_2) = (\lceil \tau_{11} \rceil, \lceil \tau_{22} \rceil). \]
This shows that when the degrees of $F_N(x, j)$ are half of the generic case, each $j$ corresponds to two points  on $X_0(N)$.
\end{example}
\section{ A variant  of  Algorithm 1}
By computing  the  list of  subresultant polynomials
of $F_N(x, j)$ and $f_N(x, J)$ with respect to the variable $x$, we obtain two integer polynomials $P_1(j, J)$ and $Q_1(j, J)$ such that $x = \frac{P_1(j, J)}{Q_1(j, J)}. $  By the same method, we obtain polynomials $P_2(j, J)$ and $Q_2(j, J)$ satisfying
$ y = \frac{P_2(j, J)}{Q_2(j, J)}. $ These polynomials appear in Kolyvagin's seminal work \cite{Kolyvagin}.  In the following, we present a variant of Algorithm 1 that directly computes these polynomials.

\begin{algorithm}[H]
\caption{Compute polynomials $P_1(j,J)$ and $Q_1(j,J)$ such that $x = \frac{P_1(j, J)}{Q_1(j, J)}$}
\begin{algorithmic}[1]
\REQUIRE  Modular functions $x$,  $j$, $J$ and their $q$-expansions.
\ENSURE Modular polynomials $P_1(j,J) = \sum_{k=0}^K \sum_{l=0}^L c_{k,l} j^k J^l$ and \\
\qquad\quad$Q_1(j,J) = \sum_{r=0}^R \sum_{s=0}^S d_{r,s} j^r J^s$.
\STATE  Select positive integers $K, L, R, S$ and set $M = \max\{(K+1)(L+1), (R+1)(S+1)\} + 50$.
\STATE For $0 \leq k \leq K$ and $0 \leq l \leq L$, compute the $q$-expansion of $j^k J^l$ to precision $M$. Let $U(k,l;n)$ denote its coefficient of $q^n$ in the expansion.
\STATE   For $0 \leq r \leq R$ and $0 \leq s \leq S$, compute the $q$-expansion of $xj^r J^s$ to precision $M$. Let $V(r,s;n)$ denote its coefficient of $q^n$ in the expansion.
\STATE  Solve the following  linear system in  variables $\{c_{k,l}\} \cup \{d_{r,s}\}$
\begin{equation}\sum_{k=0}^K \sum_{l=0}^L U(k,l;n)c_{k,l} = \sum_{r=0}^R \sum_{s=0}^S V(r,s;n)d_{r,s} \end{equation}
where $\min\{-K-NL, -R-NS-2\} \leq n \leq M$.
\STATE  If  (9.1)  has no nonzero solution, go to Step 1; otherwise
output the polynomials
\[ P_1(j,J) = \sum_{k=0}^K \sum_{l=0}^L c_{k,l} j^k J^l, \quad Q_1(j,J) = \sum_{r=0}^R \sum_{s=0}^S d_{r,s} j^r J^s. \]
\end{algorithmic}
\end{algorithm}
\begin{remark}
 (i)$K, L, R, S$  in Step 1 are generally difficult to determine theoretically and have to  be found through experimentation, starting with small values;
 (ii) The linear system of  the type $AX=BY$ in Step 4 can be solved in PARI/GP using the \texttt{matintersect} function.
\end{remark}

Applying Algorithm 6 to elliptic curves E with conductors $N \in \{11,14,17,19\}$, we obtain modular parametrizations expressed as rational functions:
\[x = \frac{P_1(j, J)}{Q_1(j, J)}, y = \frac{P_2(j, J)}{Q_2(j, J)}.\]
Substituting these expressions into (2)   produces an integer-coefficient polynomial in $j$ and $J$ that must be divisible by the modular polynomial $\Phi_N(j,J)$. This computational approach, while effective for small conductors, faces significant limitations - as pointed out  by Yang in \cite{Yang},  it is  difficult to explicitly write down these modular functions in general.

By substituting Yang-pairs $(X,Y)$ (defined in \cite{Yang,Chen})  for $(j,J)$ in Algorithm 6, we obtain computationally efficient modular parametrizations for $E$. This is illustrated below.
\begin{example}
Consider the elliptic curve $E_{a1}(46)$  defined by:
\[ y^2 + xy = x^3 - x^2 - 10x - 12. \]
The  genus  for $X_0(46)$ is $g = 5$.
Define the modular functions
\begin{align*}
X &:= \frac{1}{2}\sum_{11} \frac{E_1E_4E_{15}E_{16}E_{17}E_{18}}{E_5E_6E_7E_8E_9E_{22}} - \frac{9}{2}, \\
Y &:= \sum_{11} \frac{E_{16}E_{21}}{E_2E_7} - 2X - 19,
\end{align*}
where
\[ E_g(\tau) = q^{\left(\frac{g^2}{2N}\right)\left(\frac{g}{2}\right)\frac{N}{12}} \prod_{n=1}^{\infty} (1-q^{N(n-1)+g})(1-q^{Nn-g}), \quad q = e^{2\pi i\tau}, \]
and $\sum\limits_{11} {\prod {E_g^{{e_g}}} }$ denotes $\sum\limits_{\gamma  \in {\Gamma _0}/\Gamma } {\prod {E_g^{{e_g}}|\gamma } } $, where $\Gamma$  is  the unique cyclic subgroup of index 11 in $\Gamma_0(46)$.
Define
\begin{align*}
P_1(X,Y) &:= X^4 (-1725 - 52Y) - 42849Y^2 - 6325Y^3 - 230Y^4 \\
       &\quad + X^3 (-27692 - 3862Y - 60Y^2) \\
       &\quad + X^2 (-114264 - 36846Y - 2994Y^2 - 36Y^3) \\
       &\quad + X(-76176Y - 21114Y^2 - 1333Y^3 - 8Y^4), \\
Q_1(X,Y) &:= X^4 (-483 - 12Y) - 4761Y^2 - 207Y^3 \\
         &\quad + X^2 (-28566 - 5382Y - 270Y^2) \\
         &\quad + X^3 (-7337 - 647Y - 8Y^2) \\
         &\quad + X(-9522Y - 1863Y^2 - 36Y^3).
\end{align*}
Then $x = P_1(X,Y)/Q_1(X,Y)$.
Since ${E_{a1}}(46)$ fails to be an involutory curve [\cite{Delaunay}],
$P_1$ and $Q_1$ cannot   be obtained by Yang's construction method \cite{Yang}.
\end{example}

\section{Points of degree $d$  of   $Z_0(N)$ and  rational points on $E$}

Let $\Phi_N(X,Y)$ denote the classical modular polynomial. Define the plane curve
\[ Z_0(N) := \{(X,Y) \in \mathbb{C}^2 : \Phi_N(X,Y) = 0\}, \]
which serves as the canonical plane model for $X_0(N)$. Let $W$ be the finite set of singular points of $Z_0(N)$. The canonical map
\[ \pi: Y_0(N) \rightarrow Z_0(N), \,\,\, \tau \mapsto (j(\tau), J(\tau)) \]
 induces a bijection between $Y_0(N) \setminus \pi^{-1}(W)$ and $Z_0(N) \setminus W$.
More precisely, the fiber above a point in $W$ contain two points in $Y_0(N)$.

If  $(j, J)\in Z_0(N)$ such that  $j$ and $J$ are algebraic numbers of degree $d$,  we call $(j, J)$ a  degree $d$ point of  $Z_0(N)$.  Clearly, if $[\tau]$ is a degree $d$ point of $X_0(N)$, then $(j(\tau), J(\tau))$ is a degree $d$ point of $Z_0(N)$\cite{Ad}.

Given a  point $(j,J) \in Z_0(N)$, we can determine its corresponding $x$  by solving the equations
 $F_N(x, j)=0$ and $f_N(x, J)=0$.  As shown in previous analysis,  such $x$ is unique except in finitely many cases when the degree of $F_N(x,j)$ in $x$ is $\mu$.   In general, degree $d$ points on $Z_0(N)$ correspond to degree $d$ points on $E$.   But we have the following  proposition.

\begin{proposition}
If $(\alpha,\beta) \in E(\mathbb{Q})$, $\varphi([\tau]) = (\alpha,\beta)$ with $\tau \in \mathcal{H}$, then
\[ (j(\tau),J(\tau)) \in \overline{\mathbb{Q}}^2\,\,\, \text{with} \,\,\, \deg(j(\tau)) \leq d \,\,\, \text{and} \,\,\, \deg(J(\tau)) \leq d, \]
where $d$  is the degree of the modular parametrization.
\end{proposition}

\begin{proof}[Proof.]
If the degrees of $F_N(x,j)$ or $G_N(y,j)$ with respect to $j$ are at most $d$, the conclusion is immediate. In general, we assume that  $F_N(x,j)$ and $G_N(y,j)$ have degrees $2d$ and $3d$ in $j$,  respectively.
Let $p(j) \in \mathbb{Q}[j]$ be the minimal polynomial of $j(\tau)$. Then
$p(j) \mid F_N(\alpha,j)  \,\,\, \text{and}  \,\,\, p(j) \mid G_N(\beta,j).$
Consequently, $\deg(j) \mid \gcd(2d,3d) = d$, which implies $\deg(j) \leq d$.
The same argument shows  that $\deg(J(\tau)) \leq d$.
\end{proof}

\begin{corollary}
Let $d=1$. Then  $(\alpha,\beta) \in E(\mathbb{Q})$ if and only if
\[ (j(\tau),J(\tau)) \in \mathbb{Q}^2 \cup \{(\infty,\infty)\}. \]
\end{corollary}

\begin{remark}
Let $E(\mathbb{Q})$ be an elliptic curve with rank$(E(\mathbb{Q}))  \geqslant 1$, and fix a point $P \in E(\mathbb{Q})$ of infinite order.
The sequence
 \begin{equation}\{(j(\tau_n), J(\tau_n)) \mid \tau_n \in \varphi^{-1}(nP), n \in \mathbb{Z}_+ \} \subseteq \overline{\mathbb{Q}}^2\end{equation}
establishes that each infinite-order point  \(P\in E(\mathbb{Q}) \) corresponds to an infinite sequence of points \(\{ \Gamma_0(N)\tau_n\}_{n \geq 1} \) on \( X_0(N) \)   for which both  \( j(\tau_n) \) and \( J(\tau_n) \) are algebraic numbers of degree at most $d$.
 Typically,  these \( \tau_n \)'s  are transcendental numbers. The infinite sequence
$\{ (j(\tau_n), J(\tau_n)) \}_{n \in \mathbb{Z}_+}  $
reflects the characteristic  of the infinite-order point \( P \).
Currently,  no explicit relation is known between Heegner points on~\( X_0(N) \) and points of infinite order in~\( E(\mathbb{Q}) \) for elliptic curves~\( E \) with rank greater than~1. A potentially viable alternative approach involves relating the collection~$\{ (j(\tau_n), J(\tau_n)) \}_{n \in \mathbb{Z}^+} $ to the regulator term appearing in the BSD conjecture  for elliptic curves of rank $\geq 2$.
\end{remark}

\end{document}